\documentclass[10pt,twoside]{amsart} 
\usepackage{amsmath, amsthm, amscd, amsfonts, amssymb, graphicx, color}
\usepackage[bookmarksnumbered, colorlinks, plainpages]{hyperref}

\newtheorem{theorem}{Theorem}[section]

\newtheorem{proposition}[theorem]{Proposition}
\newtheorem{corollary}[theorem]{Corollary}

\theoremstyle{definition}
\newtheorem{definition}[theorem]{Definition}
\newtheorem{example}[theorem]{Example}

\newtheorem{remark}[theorem]{Remark}

\numberwithin{equation}{section}

\def\<{\langle}
\def\>{\rangle}

\long\def\alert#1{\smallskip{\hskip\parindent\vrule%
\vbox{\advance\hsize-2\parindent\hrule\smallskip\parindent.4\parindent%
\narrower\noindent#1\smallskip\hrule}\vrule\hfill}\smallskip}

\begin{document}
\title[Weakly classical prime submodules]{Weakly classical prime submodules}
\author[Mostafanasab, Tekir and Hakan Oral]{Hojjat Mostafanasab, \"{U}nsal Tekir and K\"{u}r\c{s}at Hakan Oral}

\subjclass[2010]{Primary: 13A15; secondary: 13C99; 13F05}
\keywords{Weakly prime submodule, Classical prime submodule, Weakly classical prime submodule.}

\begin{abstract}
In this paper, all rings are commutative with nonzero identity.
Let $M$ be an $R$-module. A proper submodule $N$ of $M$ 
is called a {\it classical prime submodule}, if for each  
$m\in M$ and elements $a,b\in R$, $abm\in N$ implies
that $am\in N$ or $bm\in N$. We introduce the 
concept of ``weakly classical prime submodules''. A proper submodule
$N$ of $M$ is a {\it weakly classical prime submodule} if whenever
$a,b\in R$ and $m\in M$ with $0\neq abm\in N$,
then $am\in N$ or $bm\in N$.
\end{abstract}
\maketitle

\section{Introduction} \label{sect1}
Throughout this paper all rings are commutative with nonzero identity and 
all modules are considered to be unitary. Several authors have extended the notion
of prime ideals to modules, see, for example \cite{D,L,MM}. Let $M$ be a module
over a commutative ring $R$. A proper submodule $N$ of $M$ is called \textit{prime} 
if for $a\in R$ and $m\in M$, $am\in N$ implies that $m\in N$ or $a\in (N:_{R}M)=\{r\in R\mid rM\subseteq N\}$. 
Anderson and Smith \cite{AS} said that a proper ideal $I$ of a ring $R$ is {\it weakly prime}
if whenever $a,b\in R$ with $0\neq ab\in I$, then $a\in I$ or $b\in I$.
Weakly prime submodules were introduced by Ebrahimi Atani and Farzalipour in \cite{E1}. A proper
submodule $N$ of $M$ is called \textit{weakly prime} if for $a\in R$ and $m\in M$
with $0\neq am\in N$, either $m\in N$ or $a\in (N:_{R}M)$.
A proper submodule $N$ of $M$ is called a \textit{classical prime submodule}, if for
each $m\in M$ and $a,b\in R$, $abm\in N$ implies that $am\in N$ or $%
bm\in N$. This notion of classical prime submodules has been extensively
studied by Behboodi in \cite{B1,B2} (see also, \cite{BK}, in which, the
notion of classical prime submodules is named \textquotedblleft weakly prime submodules\textquotedblright).
For more information on classical prime submodules, the reader is referred to \cite{A1,A,BSh}.

The annihilator of $M$ which is denoted by ${\rm Ann}_{R}(M)$ is $(0:_{R}M)$. 
Furthermore, for every $m\in M$, $(0:_Rm)$ is denoted by ${\rm Ann}_R(m)$.
When  ${\rm Ann}_{R}(M)=0$, $M$ is called a {\it faithful $R$-module}. An $R$-module $M$ is called
a \textit{multiplication module} if every submodule $N$ of $M$ has the form $%
IM$ for some ideal $I$ of $R$, see \cite{ES}. Note that, since $I\subseteq (N:_{R}M)$ then $%
N=IM\subseteq (N:_{R}M)M\subseteq N$. So that $N=(N:_{R}M)M$.
Finitely generated faithful multiplication modules are cancellation modules 
\cite[ Corollary to Theorem 9]{S}, where an $R$-module $M$ is defined to be
a \textit{cancellation module} if $IM=JM$ for ideals $I$ and $J$ of $R$
implies $I=J$. Let $N$ and $K$ be submodules of a multiplication 
$R$-module $M$ with $N=I_{1}M$ and $K=I_{2}M$ for some ideals $I_{1}$ and $%
I_{2}$ of $R$. The product of $N$ and $K$ denoted by $NK$ is defined by $%
NK=I_{1}I_{2}M$. Then by \cite[Theorem 3.4]{Am}, the product of $N$ and $K$
is independent of presentations of $N$ and $K$. Moreover, for $m,m'\in M$, by 
$mm'$, we mean the product of $Rm$ and $Rm'$. Clearly, $NK$ is a submodule of $%
M$ and $NK\subseteq N\cap K$ (see \cite{Am}). Let $N$ be a proper submodule
of a nonzero $R$-module $M$. Then the $M$-radical of $N$, denoted by $M$-$%
\mathrm{rad}(N)$, is defined to be the intersection of all prime submodules
of $M$ containing $N$. If $M$ has no prime submodule containing $N$, then we
say $M$-$\mathrm{rad}(N)=M$. It is shown in \textrm{\cite[Theorem 2.12]{ES}}
that if $N$ is a proper submodule of a multiplication $R$-module $M$, then $%
M $-$\mathrm{rad}(N)=\sqrt{(N:_{R}M)}M$. In \cite{Q}, Quartararo et al. said that a commutative ring $R$ is a $u$%
-ring provided $R$ has the property that an ideal contained in a finite
union of ideals must be contained in one of those ideals; and a $um$-ring is
a ring $R$ with the property that an $R$-module which is equal to a finite
union of submodules must be equal to one of them. 
They show that every B$\acute{\rm e}$zout ring is a $u$-ring. Moreover, they proved that 
every Pr\"{u}fer domain is a $u$-domain. Also, any ring which contains an infinite field
as a subring is a $u$-ring, \cite[Exercise 3.63]{Sh}.

In this paper we introduce the concept of weakly classical prime
submodules. A proper submodule $N$ of an $R$-module $M$ is called a
{\it weakly classical prime submodule} if whenever
$a,b\in R$ and $m\in M$ with $0\neq abm\in N$,
then $am\in N$ or $bm\in N$. Clearly, every classical prime submodule is a weakly classical prime submodule.
Among many results in this paper, it is shown (Theorem \ref{main}) that $N$ is a weakly classical prime submodule of
an $R$-module $M$ if and only if for every ideals $I,~J$ of $R$ and $m\in M$ with $0\neq IJm\subseteq N$, either
$Im\subseteq N$ or $Jm\subseteq N$. It is proved  (Theorem \ref{T2})  that if $N$ is a weakly classical prime submodule of an 
$R$-module $M$ that is not classical prime, then $(N:_{R}M)^{2}N=0$. It is shown (Theorem \ref{main2}) that  over a $um$-ring $R$, 
$N$ is a weakly classical prime submodule of an $R$-module $M$ if and only if for every ideals $I,~J$ of $R$ and submodule
$L$ of $M$ with $0\neq IJL\subseteq N$, either $IL\subseteq N$ or $JL\subseteq N$.
Let $R=R_{1}\times R_{2}\times R_3$ be a decomposable ring and $M=M_{1}\times M_{2}\times M_3$ be an $R$-module where
$M_{i}$ is an $R_{i}$-module, for $i=1,2,3$. In Theorem \ref{product3}
it is proved that if $N$ is a weakly classical prime submodule of $M$, then either $N=\{(0,0,0)\}$ or $N$ is a classical prime submodule of $M$.
Let $R$ be a $um$-ring, $M$ be an $R$-module and $F$ be a faithfully flat $R$-module. It is shown (Theorem \ref{flat}) that $N$ is a
weakly classical prime submodule of $M$ if and only if $F\otimes N$ is a
weakly classical prime submodule of $F\otimes M.$

\bigskip

\section{Properties of weakly classical prime submodules}

First of all we give a module which has no nonzero weakly classical 
prime submodule.
\begin{example}
 Let $p$ be a fixed prime integer and $\mathbb{N}_{0}=\mathbb{N}\cup \left\{ 0\right\}.$
Then $$E\left( p\right) :=\left\{ \alpha \in \mathbb{Q}/\mathbb{Z}\mid\alpha =\frac{r}{p^{n}}+\mathbb{Z}
\text{ \ for some }r\in\mathbb{Z}\text{ and }n\in \mathbb{N}_{0}\right\}$$ is a nonzero
submodule of the $\mathbb{Z}$-module $\mathbb{Q}/\mathbb{Z}.$ For each $t\in \mathbb{N}_{0},$ set $$G_{t}:=\left\{
\alpha \in \mathbb{Q}/\mathbb{Z}\mid\alpha =\frac{r}{p^{t}}+\mathbb{Z}\text{ \ for some }r\in \mathbb{Z}\right\}.$$ 
Notice that for each $t\in \mathbb{N}_{0}$, $G_{t}$ is a submodule of $E\left( p\right) $ generated by $\frac{1}{%
p^{t}}+\mathbb{Z}$ for each $t\in \mathbb{N}_{0}.$ Each proper submodule of $E\left(
p\right) $ is equal to $G_{i}$ for some $i\in \mathbb{N}_{0}\left( \text{see, \cite[Example 7.10]{Sh}}%
\right).$ However, no $G_{t}$ is a weakly classical prime submodule of $%
E\left( p\right).$ Indeed, $\frac{1}{p^{t+2}}+\mathbb{Z}\in E\left(
p\right) $. Then $0\neq p^{2}\left( \frac{1}{p^{t+2}}+\mathbb{Z}\right) =\frac{1}{p^{t}}%
+\mathbb{Z}\in G_{t}$  but $p\left( \frac{1}{p^{t+2}}+\mathbb{Z}\right) =\frac{1}{%
p^{t+1}}+\mathbb{Z}\notin G_{t}$. 
\end{example}

\begin{theorem}\label{colon}
Let $M$ be an $R$-module and $N$ a proper submodule of $M$. 
\begin{enumerate}
\item If $N$ is a weakly classical prime submodule of $M$, then $\left( N:_Rm\right) $ is a weakly 
prime ideal of $R$ for every $m\in M\backslash N$ with $Ann_R(m)=0$. 

\item If $\left( N:_Rm\right) $ is a weakly prime ideal of $R$ for every $m\in M\backslash N$, then $N$ 
is a weakly classical prime submodule of $M$.
\end{enumerate}
\end{theorem}

\begin{proof}
(1) Suppose that $N$ is a weakly classical prime submodule. Let $m\in M\backslash N$ with ${\rm Ann}_R(m)=0$, 
and $0\neq ab\in \left( N:_Rm\right)$ for some $a,b\in R$. Then $0\neq abm\in N$. So $am\in N$ or $bm\in N$. 
Hence $a\in \left( N:_Rm\right)$ or $b\in \left( N:_Rm\right) $.

(2) Assume that $\left( N:_Rm\right) $ is a weakly prime ideal of $R$ for every $m\in M\backslash N$.
Let $0\neq abm\in N$ for some $m\in M$ and $a,b\in R$. If $m\in N$, then we are done. So
we assume that $m\notin N$. Hence $0\neq ab\in(N:_Rm)$ implies that either $a\in (N:_Rm)$
or $b\in (N:_Rm)$. Therefore either $am\in N$ or $bm\in N$, and so $N$ is weakly classical prime.
\end{proof}

We recall that $M$ is a torsion-free $R$-module if and only if for every $0\neq m\in M$, ${\rm Ann}_R(m)=0$.
As a direct consequence of Theorem \ref{colon} the following result follows.
\begin{corollary}
Let $\ M$ be a torsion-free $R$-module and $N$ a proper submodule of $M$. Then 
$N$ is a weakly classical prime submodule of $M$ if and only if $\left( N:_Rm\right) $ is a weakly prime 
ideal of $R$ for every $m\in M\backslash N$.
\end{corollary}

\begin{theorem}
\label{im} Let $f:M\to M^{\prime}$ be a homomorphism of $R$-modules.
\begin{enumerate}
\item Suppose that $f$ is a monomorphism. 
If $N^{\prime}$ is a weakly classical prime submodule of $M^{\prime}$ with $f^{-1}(N^{\prime })\neq M$,
then $f^{-1}(N^{\prime})$ is a weakly classical prime submodule of $M$.
\item Suppose that $f$ is an epimorphism. If $N$ is a weakly classical prime submodule of $M$
containing $Ker(f)$, then $f(N)$ is a weakly classical prime submodule of $%
M^{\prime}$.
\end{enumerate}
\end{theorem}

\begin{proof}
$(1)$ Suppose that $N^{\prime }$ is a weakly classical prime submodule of $M^{\prime }$ with $f^{-1}(N^{\prime })\neq M$.
Let $0\neq abm\in f^{-1}(N^{\prime })$ for some $a,b\in R$ and $m\in M$. Since $f$ is a monomorphism, 
$0\neq f\left( abm\right) \in N^{\prime }$. So we get $0\neq abf(m)\in
N^{\prime }$. Hence $f(am)=af(m)\in N^{\prime }$ or $f(bm)=bf(m)\in
N^{\prime }$. Thus $am\in f^{-1}(N^{\prime })$ or $bm\in f^{-1}(N^{\prime })$.
Therefore   $f^{-1}(N^{\prime})$ is a weakly classical prime submodule of $M$.

(2) Assume that $N$ is a weakly classical prime submodule of $M$. 
Let $a,b\in R$ and $m^{\prime}\in M^{\prime}$ be such that $0\neq abm^{\prime}\in f(N)$.
By assumption there exists $m\in M$ such that $m^{\prime}=f(m)$ and so $%
f(abm)\in f(N)$. Since $Ker(f)\subseteq N$, we have $0\neq abm\in N$. It implies
that $am\in N$ or $bm\in N$. Hence $am^{\prime}\in f(N)$ or 
$bm^{\prime}\in f(N)$. Consequently $f(N)$ is a weakly classical prime submodule of $M^{\prime}$.
\end{proof}

As an immediate consequence of Theorem \ref{im}(2) we have the following
corollary.

\begin{corollary}\label{quo} 
Let $M$ be an $R$-module and $L\subset N$ be submodules of $M$.
If $N$ is a weakly classical prime submodule of $M$, then $N/L$ is a
weakly classical prime submodule of $M/L$.
\end{corollary}

\begin{theorem}
Let $K$ and $N$ be submodules of $M$ with $K\subset N\subset M$. If $K$ is a
weakly classical prime submodule of $M$ and $N/K$ is a weakly classical prime 
submodule of $M/K$, then $N$ is a weakly classical prime submodule of $M$.
\end{theorem}

\begin{proof}
Let $a,b\in R$, $m\in M$ and $0\neq abm\in N$. If $abm\in K$, then $am\in K\subset N$ 
or $bm\in K\subset N$ as it is needed. Thus, assume that $abm\not\in K$. Then 
$0\neq ab(m+K)\in N/K$, and so $a(m+K)\in N/K$ or $b(m+K)\in N/K$. It means that 
$am\in N$ or $bm\in N$, which completes the proof.
\end{proof}

For an $R$-module $M$, the set of zero-divisors of $M$ is denoted by $Z_R(M)$.
\begin{theorem}
Let $M$ be an $R$-module, $N$ be a submodule and $S$ be a multiplicative
subset of $R$. 
\begin{enumerate}
\item If $N$ is a weakly classical prime submodule of $M$ such that $\left( N:_RM\right)\cap S=\emptyset$,
then $S^{-1}N$ is a weakly classical prime submodule of $S^{-1}M$.
\item If $S^{-1}N$ is a weakly classical prime submodule of $S^{-1}M$ such that $S\cap Z_R(N)=\emptyset$
and $S\cap Z_R(M/N)=\emptyset$, then $N$ is a weakly classical prime submodule of $M$.
\end{enumerate}
\end{theorem}

\begin{proof}
(1) Let $N$ be a weakly classical prime submodule of $M$ and 
$\left( N:_RM\right) \cap S=\emptyset$. Suppose that 
$0\neq\frac{a_{1}}{s_{1}}\frac{a_{2}}{s_{2}}\frac{m}{s_{3}}\in S^{-1}N$
for some $a_1,a_2\in R$, $s_1,s_2,s_3\in S$ and $m\in M$.
Then there exists $s\in S$ such that $sa_1a_2m\in N$. 
If $sa_1a_2m=0$, then $\frac{a_{1}}{s_{1}}\frac{a_{2}}{s_{2}}\frac{m}
{s_{3}}=\frac{sa_1a_2m}{ss_1s_2s_3}=\frac{0}{1}$, a contradiction.
Since $N$ is a weakly classical prime submodule, then we have $a_{1}\left(
sm\right)\in N$ or $a_{2}\left( sm\right) \in N$. Thus 
$\frac{a_{1}}{s_{1}}\frac{m}{s_{3}}=\frac{sa_{1}m}{ss_{1}s_3}\in S^{-1}N$ or
$\frac{a_{2}}{s_{2}}\frac{m}{s_{3}}=\frac{sa_{2}m}{ss_{2}s_3}\in S^{-1}N$.
Consequently $S^{-1}N$ is a weakly classical prime submodule of $S^{-1}M$.

(2) Suppose that $S^{-1}N$ is a weakly classical prime submodule of $S^{-1}M$ and $S\cap Z_R(N)=\emptyset$
and $S\cap Z_R(M/N)=\emptyset$. Let $a,b\in R$ and $m\in M$ such that $0\neq abm\in N$. Then $\frac{a}{1}\frac{b}{1}\frac{m}{1}\in S^{-1}N$. If $\frac{a}{1}\frac{b}{1}\frac{m}{1}=\frac{0}{1}$, then there exists $s\in S$ such that $sabm=0$
which contradicts $S\cap Z_R(N)=\emptyset$. Therefore $\frac{a}{1}\frac{b}{1}\frac{m}{1}\neq\frac{0}{1}$, and so either $\frac{a}{1}\frac{m}{1}\in S^{-1}N$ or $\frac{b}{1}\frac{m}{1}\in S^{-1}N$. We may assume that $\frac{a}{1}\frac{m}{1}\in S^{-1}N$. So there exists $u\in S$ such that $uam\in N$. But $S\cap Z_R(M/N)=\emptyset$, whence $am\in N$. Consequently $N$ is a weakly classical prime submodule of $M$.
\end{proof}

Darani \cite{YS} generalized the concept of prime submodules (resp. weakly prime
submodules) of a module over a commutative ring as
following: Let $N$ be a proper submodule of an $R$-module $M$. Then $N$ is said
to be a \textit{2-absorbing submodule (resp.  weakly 2-absorbing 
submodule) of} $M$ if whenever $a,b\in R$ and $m\in M$ with $abm\in N$ (resp. 
$0\neq abm\in N$), then $am\in N$ or $bm\in N$ or $ab\in (N:_{R}M)$.
\begin{proposition}\label{abs-class}
Let $N$ be a proper submodule of an $R$-module $M$. 
\begin{enumerate}
\item If $N$ is a weakly prime submodule of $M$, then $N$ is a weakly classical prime submodule of $M$.

\item If $N$ is a weakly classical prime submodule of $M$, then $N$ is a weakly 2-absorbing submodule of $M$. The 
converse holds if in addition $(N:_RM)$ is a weakly prime ideal of $R$.
\end{enumerate}
\end{proposition}
\begin{proof}
(1) Assume that $N$ is a weakly prime submodule of $M$. Let $a,b\in R$ and $m\in M$
such that $0\neq abm\in N$. Therefore either $bm\in N$ or $a\in(N:_RM)$.
The first case leads us to the claim. In the second case we have that $am\in N$.
Consequently $N$ is a weakly classical prime submodule.

(2) It is evident that if $N$ is weakly classical prime, then it is weakly 2-absorbing. 
Assume that $N$ is a weakly 2-absorbing submodule of $M$ and $(N:_RM)$ is a weakly prime ideal of $R$. 
Let $0\neq abm\in N$ for some $a,b\in R$ and $m\in M$ such that neither $am\in N$ nor $bm\in N$. 
Then $0\neq ab\in(N:_RM)$ and so either $a\in(N:_RM)$ or $b\in(N:_RM)$. This contradiction 
shows that $N$ is weakly classical prime.
\end{proof}

The following example shows that the converse of Proposition \ref{abs-class}(1) is not true.
\begin{example}
Let $R=\mathbb{Z}$ and $M=\mathbb{Z}_p\bigoplus\mathbb{Z}\bigoplus\mathbb{Z}$ 
where $p$ is a prime integer. Consider the submodule $N=\{\overline{0}\}\bigoplus\{0\}\bigoplus\mathbb{Z}$
of $M$. Notice that $(\overline{0},0,0)\neq p(\overline{1},0,1)=(\overline{0},0,p)\in N$, but 
$(\overline{1},0,1)\notin N$. Also $p(\overline{1},1,1)\notin N$ which shows that $p\notin(N:_{\mathbb{Z}}M)$. Therefore $N$
is not a weakly prime submodule of $M$. Now, we show that $N$ is a
weakly classical prime submodule of $M$. Let $m,n,z,w\in \mathbb{Z}$ and $\overline{x}\in\mathbb{Z}_p$
be such that $(\overline{0},0,0)\neq mn(\overline{x},z,w)\in N$. Hence $\overline{mnx}=\overline{0}$ and $mnz=0$.
Therefore $p|mnx$ and $z=0$. So $p| m$ or $p|nx$. If $p|m$, then 
$m(\overline{x},z,w)=(\overline{mx},0,mw)=(\overline{0},0,mw)\in N$. Similarly, if
  $p|nx$, then $n(\overline{x},z,w)=(\overline{nx},0,nw)=(\overline{0},0,nw)\in N$.
 Consequently $N$ is a weakly classical prime submodule of $M$.
\end{example}

\begin{proposition}
Let $M$ be a cyclic $R$-module. Then a proper submodule $N$ of $M$ 
is a weakly prime submodule if and only if it is a weakly classical prime submodule.

\end{proposition}

\begin{proof}
By Proposition \ref{abs-class}(1), the ``only if'' part holds.
Let $M=Rm$ for some $m\in M$ and $N$ be a weakly classical prime submodule of $M$. Suppose that 
$0\neq rx\in N$ for some $r\in R$ and $x\in M$. Then there exists an element $%
s\in R$ such that $x=sm$. Therefore $0\neq rx=rsm\in N$ and since $N$ is a
weakly classical prime submodule, $rm\in N$ or $sm\in N$. Hence $r\in (N:_RM)$
or $x\in N$. Consequently $N$ is a weakly prime submodule.
\end{proof}

\begin{example}
Let $R=\mathbb{Z}$ and $M=\mathbb{Z}_p\bigoplus\mathbb{Z}_q\bigoplus\mathbb{Q}$ 
where $p,~q$ are two distinct prime integers. One can easily see that the zero submodule of $M$
is a weakly classical prime submodule. Notice that $pq(\overline{1},\overline{1},0)=(\overline{0},\overline{0},0)$, but $p(\overline{1},\overline{1},0)\not=(\overline{0},\overline{0},0)$ and
$q(\overline{1},\overline{1},0)\not=(\overline{0},\overline{0},0)$. So the zero submodule of $M$ is not classical prime. 
Hence the two concepts of classical prime submodules and of weakly classical prime submodules
are different in general.
\end{example}

\begin{definition}
Let $N$ be a proper submodule of $M$ and $a,b\in R,$ $m\in M.$
If $N$ is a weakly classical prime submodule and $abm=0,$ $am\notin N$, $bm\notin N$, then $(a,b,m)$ is called a 
{\it classical triple-zero of} $N$.
\end{definition}

\begin{theorem}\label{le1}Let $N$ be a weakly classical prime submodule of an $R$-module $M$ and
suppose that $abK\subseteq N$ for some $a,b\in R$ and some submodule $K$ of 
$M$. If $(a,b,k)$ is not a classical triple-zero of $N$ for every $k\in K$, then $aK\subseteq N$ or
$bK\subseteq N$.
\end{theorem}

\begin{proof}
Suppose that $(a,b,k)$ is not a classical triple-zero of $N$ for every $k\in K$. 
Assume on the contrary that $aK\not\subseteq N$ and $bK\not\subseteq N$. Then there are 
$k_{1},k_{2}\in K$ such that $ak_{1}\not\in N$ and $bk_{2}\not\in N$. 
If $abk_{1}\neq0$, then we have $bk_{1}\in N$, because 
$ak_{1}\not\in N$ and $N$ is a weakly classical prime submodule
of $M$. If $abk_{1}=0,$ then since $ak_{1}\notin N$ and $(a,b,k_{1})$ 
is not a classical triple-zero of $N$, we conclude again that $bk_{1}\in N$. By a
similar argument, since $(a,b,k_2)$ is not a classical triple-zero and $bk_2\notin N$, then we deduce that
$ak_{2}\in N$. From our hypothesis, $ab(k_{1}+k_{2})\in N$
and $(a,b,k_{1}+k_{2})$ is not a classical triple-zero of $N$.
Hence we have either $a(k_{1}+k_{2})\in N$ or $b(k_{1}+k_{2})\in N$. 
If $a(k_{1}+k_{2})=ak_{1}+ak_{2}\in N$, then since $%
ak_{2}\in N$, we have $ak_{1}\in N$, a contradiction. If 
$b(k_{1}+k_{2})=bk_{1}+bk_{2}\in N$, then since $bk_{1}\in N$,
we have $bk_{2}\in N$, which again is a contradiction. Thus 
$aK\subseteq N$ or $bK\subseteq N$.
\end{proof}

\begin{definition}
Let $N$ be a weakly classical prime submodule of an $R$-module $M$ and suppose that 
$IJK\subseteq N$ for some ideals $I,$ $J$ of $R$ and some submodule $K$ of 
$M$. We say that $N$ is a \textit{free }\textit{classical triple-zero with respect to 
}$IJK$ if $(a,b,k)$ is not a classical triple-zero of $N$ for every $a\in
I,b\in J$, and $k\in K$.
\end{definition}

\begin{remark}
Let $N$ be a weakly classical prime submodule of $M$ and suppose that $%
IJK\subseteq N$ for some ideals $I,J$ of $R$ and some submodule $K$ of $M$
such that $N$ is a free classical triple-zero with respect to $IJK$. Hence, if $%
a\in I,b\in J$, and $k\in K$, then $ak\in N$ or $bk\in N$.
\end{remark}

\begin{corollary}
\label{main} Let $N$ be a weakly classical prime submodule of an $R$-module $M$ and
suppose that $IJK\subseteq N$ for some ideals $%
I,J$ of $R$ and some submodule $K$ of $M$. If $N$ is a free classical 
triple-zero with respect to $IJK$, then $IK\subseteq N$ or $JK\subseteq N$.
\end{corollary}

\begin{proof}
Suppose that $N$ is a free classical triple-zero with respect to $IJK$. 
Assume that $IK\not\subseteq N$ and 
$JK\not\subseteq N$. Then there are $a\in I$ and $b\in J$
with $aK\not\subseteq N$ and $bK\not\subseteq N$.
Since $abK\subseteq N$ and $N$ is free classical 
triple-zero with respect to $IJK$, then Theorem \ref{le1} implies that $aK\subseteq N$ and $bK\subseteq N$
which is a contradiction. Consequently $IK\subseteq N$ or $JK\subseteq N$.
\end{proof}

Let $M$ be an $R$-module and $N$ a submodule of $M$. For every $a\in R$, $\{m\in M\mid am\in N\}$ is denoted by $(N:_M a)$. It is easy to see that $(N:_Ma)$ is a submodule of $M$ containing
$N$.

In the next theorem we characterize weakly classical prime submodules.
\begin{theorem}\label{main}
Let $M$ be an $R$-module and $N$ be a proper submodule of $M$.
The following conditions are equivalent:
\begin{enumerate}
\item $N$ is weakly classical prime;

\item For every $a,b\in R$, $(N:_Mab)=(0:_Mab)\cup(N:_Ma)\cup(N:_Mb)$;

\item For every $a\in R$ and $m\in M$ with $am\notin N$, $(N:_Ram)=(0:_Ram)\cup(N:_Rm)$;

\item For every $a\in R$ and $m\in M$ with $am\notin N$, $(N:_Ram)=(0:_Ram)$
or $(N:_Ram)=(N:_Rm)$;

\item For every $a\in R$ and every ideal $I$ of $R$ and $m\in M$ with $0\neq aIm\subseteq N$, 
either $am\in N$ or $Im\subseteq N$;

\item For every ideal $I$ of $R$ and $m\in M$ with $Im\nsubseteq N$, 
$(N:_RIm)=(0:_RIm)$ or $(N:_RIm)=(N:_Rm)$;

\item For every ideals $I,~J$ of $R$ and $m\in M$ with $0\neq IJm\subseteq N$, either
$Im\subseteq N$ or $Jm\subseteq N$.
\end{enumerate}
\end{theorem}
\begin{proof}
(1)$\Rightarrow$(2) Suppose that $N$ is a weakly classical prime submodule of $M$. Let $m\in
( N:_Mab)$. Then $abm\in N$. If $abm=0$, then $m\in( 0:_Mab)$. Assume that $abm\neq0$. 
Hence $am\in N$ or $bm\in N$. Therefore $m\in\left( N:_Ma\right)$ or $m\in\left( N:_Mb\right)$. 
Consequently, $\left(N:_Mab\right)=(0:_Mab)\cup\left(N:_Ma\right)\cup\left( N:_Mb\right)$.\newline
(2)$\Rightarrow$(3) 
Let $am\notin N$ for some $a\in R$ and $m\in M$. Assume that $x\in(N:_Ram)$.
Then $axm\in N$, and so $m\in(N:_Max)$. Since $am\notin N$, then $m\notin(N:_Ma)$.
Thus by part (2), $m\in(0:_Max)$ or $m\in(N:_Mx)$, whence $x\in(0:_Ram)$
or $x\in(N:_Rm)$. Therefore $(N:_Ram)=(0:_Ram)\cup(N:_Rm)$.\newline
(3)$\Rightarrow$(4) By the fact that if an ideal (a subgroup) is the union
of two ideals (two subgroups), then it is equal to one of them.\newline
(4)$\Rightarrow$(5) Let for some $a\in R$, an ideal $I$ of $R$ and $m\in M$, $0\neq aIm\subseteq N$.
Hence $I\subseteq(N:_Ram)$ and $I\nsubseteq(0:_Ram)$.
If $am\in N$, then we are done. So, assume that $am\notin N$.
Therefore by part (4) we have that  $I\subseteq(N:_Rm)$, i.e., $Im\subseteq N$.\newline
(5)$\Rightarrow$(6)$\Rightarrow$(7) Have proofs similar to that of the previous implications.\newline
(7)$\Rightarrow$(1) Is trivial.
\end{proof}

\begin{theorem}\label{T1} 
Let $N$ be a weakly classical prime submodule of $M$ and suppose that $(a,b,m)$ 
is a classical triple-zero of $N$ for some $a,b\in R,$ $m\in M$. Then

\begin{enumerate}
\item $abN=0.$

\item $a(N:_RM)m=0.$

\item $b(N:_RM)m=0.$

\item $(N:_RM)^2m=0.$

\item $a(N:_RM)N=0.$

\item $b(N:_RM)N=0.$
\end{enumerate}
\end{theorem}

\begin{proof}
(1) Suppose that $abN\neq 0$. Then there exists $n\in N$
with $abn\neq0$. Hence $0\neq ab(m+n)=abn\in N$, so we conclude that $a(m+n)\in
N $ or $b(m+n)\in N$. Thus $am\in N$ or $bm\in N$, which contradicts the assumption that $(a,b,m)$ is classical triple-zero.
Thus $abN=0.$

(2) Let $axm\neq0$ for some $x\in (N:_{R}M)$. Then $a(b+x)m\neq0$, because $abm=0$. 
Since $xm\in N$, $a(b+x)m\in N$. Then $am\in N$ or $(b+x)m\in N$. Hence 
$am\in N$ or $bm\in N$, which contradicts our hypothesis. 

(3) The proof is similar to part (2).

(4) Assume that $x_{1}x_{2}m\neq0$ for some $x_{1},x_{2}\in
(N:_{R}M)$. Then by parts (2) and (3), $(a+x_{1})(b+x_{2})m=x_1x_2m\neq0$.
Clearly $(a+x_{1})(b+x_{2})m\in N$. Then $(a+x_{1})m\in N$ or $(b+x_{2})m\in
N$. Therefore $am\in N$ or $bm\in N$ which is a contradiction. Consequently $(N:_{R}M)^{2}m=0.$

(5) Let $axn\neq0$ for some $x\in(N:_RM)$ and $n\in N$. Therefore by parts (1) and (2) we conclude that
$0\neq a(b+x)(m+n)=axn\in N$. So $a(m+n)\in N$ or $(b+x)(m+n)\in N$. Hence
$am\in N$ or $bm\in N$. This contradiction shows that $a(N:_RM)N=0.$

(6) Similart to part (5).
\end{proof}

A submodule $N$ of an $R$-module $M$ is called a nilpotent submodule if $(N
:_RM)^kN=0$ for some positive integer $k$ (see \cite{Ali}), and we say that $%
m\in M$ is nilpotent if $Rm$ is a nilpotent submodule of $M$.

\begin{theorem}\label{T2} 
If $N$ is a weakly classical prime submodule of an $R$-module $M$ that is
not classical prime, then $(N:_{R}M)^{2}N=0$ and so $N$ is nilpotent.
\end{theorem}

\begin{proof}
Suppose that $N$ is a weakly classical prime submodule of $M$ that is
not classical prime. Then there exists a classical triple-zero $%
(a,b,m)$ of $N$ for some $a,b\in R,$ $m\in M$. Assume that $%
(N:_{R}M)^{2}N\neq0$. Hence there are $x_{1},x_{2}\in
(N:_{R}M) $ and $n\in N$ such that $x_{1}x_{2}n\neq0$. By Theorem %
\ref{T1}, $0\neq (a+x_{1})(b+x_{2})(m+n)=x_{1}x_{2}n\in N$. So $%
(a+x_{1})(m+n)\in N$ or $(b+x_{1})(m+n)\in N$. Therefore $am\in N$ 
or $bm\in N$, a contradiction.
\end{proof}

\begin{remark}\label{multi}
Let $M$ be a multiplication $R$-module and $K,L$ be submodules of $M$.
Then there are ideals $I,J$ of $R$ such that $K=IM$ and $L=JM$. Thus
$KL=IJM=IL$. In particular $KM=IM=K$. Also, for any $m\in M$ we define
$Km:=KRm$. Hence $Km=IRm=Im$. 
\end{remark}
\begin{corollary}
If $N$ is a weakly classical prime submodule of a multiplication $R$%
-module $M$ that is not classical prime, then $N^{3}=0$.
\end{corollary}
\begin{proof}
Since $M$ is multiplication, then $N=(N:_RM)M$. Therefore by Theorem \ref{T2} and Remark \ref{multi}, $N^3=(N:_RM)^2N=0$.
\end{proof}

Assume that $\mathrm{Nil}(M)$ is the set of nilpotent elements of $M$. If $M$
is faithful, then $\mathrm{Nil}(M)$ is a submodule of $M$ and if $M$ is
faithful multiplication, then $\mathrm{Nil}(M)=\mathrm{Nil}(R)M=\bigcap Q$ ($=M$-$\rm rad(\{0\})$),
where the intersection runs over all prime submodules of $M$, \textrm{\cite[%
Theorem 6]{Ali}}. 
\begin{theorem}
\label{nil} Let $N$ be a weakly classical prime submodule of $M$. If $N$
is not classical prime, then

\begin{enumerate}
\item $\sqrt{(N:_RM)}=\sqrt{Ann_R(M)}$.

\item If $M$ is multiplication, then $M$-$\mathrm{rad}(N)$=$M$-$\mathrm{rad}(\{0\})$.
If in addition $M$ is faithful, then $M$-$\mathrm{rad}(N)=\mathrm{Nil}(M)$.
\end{enumerate}
\end{theorem}

\begin{proof}
(1) Assume that $N$ is not classical prime. By Theorem \ref{T2}, $%
(N:_RM)^2N=0$. Then 
\begin{equation*}
(N:_RM)^3=(N:_RM)^2(N:_RM)
\end{equation*}
\begin{equation*}
\hspace{2cm}\subseteq((N:_RM)^2N:_RM)
\end{equation*}
\begin{equation*}
=(0:_RM),
\end{equation*}
and so $(N:_RM)\subseteq\sqrt{(0:_RM)}$. Hence, we have $\sqrt{(N:_RM)}=
\sqrt{(0:_RM)}=\sqrt{{\rm Ann}_R(M)}$.

(2) By part (1), $M $-$\mathrm{rad}(N)=\sqrt{(N:_{R}M)}M=\sqrt{(0:_RM)}M=M$-$\mathrm{rad}(\{0\})={\rm Nil}(M)$.
\end{proof}

\begin{corollary}
Let $R$ be a ring and $I$ be a proper ideal of $R$.
\begin{enumerate}
\item  $_{R}I$ is a weakly classical prime submodule of $_{R}R$ if and only if $I$ is a weakly prime ideal of $R$. 

\item Every proper ideal of $R$ is weakly prime if and only if for every $R$-module $M$
and every proper submodule $N$ of $M$, $N$ is a weakly classical prime submodule of $M$.
\end{enumerate} 
\end{corollary}
\begin{proof}
(1) Let $_{R}I$ be a weakly classical prime submodule of $_{R}R$. Then by Theorem \ref{colon}(1),
$(I:_R1)=I$ is a weakly prime ideal of $R$. For the converse, notice that $_{R}I$ is a weakly prime submodule of $_{R}R$ if and only if $I$ is a weakly prime ideal of $R$. Now, apply part (1) of Proposition \ref{abs-class}.

(2) Assume that every proper ideal of $R$ is weakly prime.
Let $N$ be a proper submodule of an $R$-module $M$.
Since for every $m\in M\backslash N$, $(N:_Rm)$ is a proper ideal of $R$, then it is a weakly prime ideal of $R$. 
Hence by Theorem \ref{colon}(2), $N$ is a weakly classical prime submodule of $M$. We have the converse immediately 
by part (1).
\end{proof}

Regarding Remark \ref{multi} we have the next proposition.
\begin{proposition}
Let  $M$ be a multiplication $R$-module and $N$ be a proper submodule of $M$.
The following conditions are equivalent:

\begin{enumerate}
\item $N$ is a weakly classical prime submodule of $M$;

\item If $0\neq N_1N_2m\subseteq N$ for some submodules $N_1,N_2$ of $
M$ and $m\in M$, then either $N_1m\subseteq N$ or $N_2m\subseteq N$.
\end{enumerate}
\end{proposition}

\begin{proof}
(1)$\Rightarrow$(2) Let $0\neq N_1N_2m\subseteq N$ for some submodules $N_1,N_2$
of $M$ and $m\in M$. Since $M$ is multiplication, there are ideals $I_1,I_2$ of $R$ such that $N_1=I_1M$
and $N_2=I_2M$. Therefore $0\neq N_1N_2m= I_1I_2m\subseteq N$, and 
so either $I_1m\subseteq N$ or $I_2m\subseteq N$.
Hence $N_1m\subseteq N$ or $N_2m\subseteq N$.\newline
(2)$\Rightarrow$(1) Suppose that $0\neq I_{1}I_{2}m\subseteq N$ for some ideals $%
I_{1},I_{2}$ of $R$ and some $m\in M$. It is sufficient to set $%
N_1:=I_1M$ and $N_2:=I_2M$ in part (2).
\end{proof}

\begin{theorem}\label{main2}
Let $R$ be a $um$-ring, $M$ be an $R$-module and $N$ be a proper submodule of $M$.
The following conditions are equivalent:
\begin{enumerate}
\item $N$ is weakly classical prime;

\item For every $a,b\in R$, $(N:_Mab)=(0:_Mab)$ or $(N:_Mab)=(N:_Ma)$
or $(N:_Mab)=(N:_Mb)$;

\item For every $a,b\in R$ and every submodule $L$ of $M$, $0\neq abL\subseteq N$
implies that $aL\subseteq N$ or $bL\subseteq N$;

\item For every $a\in R$ and every submodule $L$ of $M$ with $aL\nsubseteq N$, $(N:_RaL)=(0:_RaL)$
or $(N:_RaL)=(N:_RL)$;

\item For every $a\in R$, every ideal $I$ of $R$ and every submodule $L$ of $M$, $0\neq aIL\subseteq N$
implies that $aL\subseteq N$ or $IL\subseteq N$;

\item For every ideal $I$ of $R$ and every submodule $L$ of $M$ with $IL\nsubseteq N$, $(N:_RIL)=(0:_RIL)$
or $(N:_RIL)=(N:_RL)$;

\item For every ideals $I,~J$ of $R$ and every submodule $L$ of $M$, $0\neq IJL\subseteq N$
implies that $IL\subseteq N$ or $JL\subseteq N$.
\end{enumerate}
\end{theorem}
\begin{proof}
Similar to that of Theorem \ref{main}.
\end{proof}

\begin{remark}
The zero submodule of the $\mathbb{Z}$-module $\mathbb{Z}_4$, is a weakly classical prime submodule (weakly prime ideal) of $\mathbb{Z}_4$, but $(0:_{\mathbb{Z}}\mathbb{Z}_4)=4\mathbb{Z}$ is not a weakly prime ideal of $\mathbb{Z}$.
\end{remark}

\begin{proposition}\label{faith}
Let $R$ be a $um$-ring, $M$ be an $R$-module and $N$ be a proper submodule of $M$.
If $N$ is a weakly classical prime submodule of $M$, then $(N:_RL)$ is a weakly prime ideal of $R$
for every faithful submodule $L$ of $M$ that is not contained in $N$.
\end{proposition}
\begin{proof}
Assume that $N$ is a weakly classical prime submodule of $M$ and $L$ is a faithful submodule of $M$ 
such that $L\nsubseteq N$. Let $0\neq ab\in(N:_RL)$ for some $a,b\in R$. Then $0\neq abL\subseteq N$, because $L$ is faithful.
Hence Theorem \ref{main2} implies that $aL\subseteq N$ or $bL\subseteq N$, i.e., $a\in(N:_RL)$ or $b\in(N:_RL)$. Consequently
$(N:_RL)$ is a weakly prime ideal of $R$. 
\end{proof}

\begin{proposition}
Let $M$ be an $R$-module and $N$ be a weakly classical prime submodule of $M$.Then
\begin{enumerate}
\item For every $a,b\in R$ and $m\in M$, $(N:_Rabm)=(0:_Rabm)\cup(N:_Ram)\cup(N:_Rbm)$;

\item If $R$ is a $u$-ring, then for every $a,b\in R$ and $m\in M$,
$(N:_Rabm)=(0:_Rabm)$ or $(N:_Rabm)=(N:_Ram)$ or $(N:_Rabm)=(N:_Rbm)$.
\end{enumerate}
\end{proposition}
\begin{proof}
(1) Let $a,b\in R$ and $m\in M$. Suppose that $r\in(N:_Rabm)$. Then $ab(rm)\in N$. 
If $ab(rm)=0$, then $r\in(0:_Rabm)$. Therefore we assume that $ab(rm)\neq0$.
So, either $a(rm)\in N$ or $b(rm)\in N$. Thus, either $r\in(N:_Ram)$ or $r\in(N:_Rbm)$. Consequently
$(N:_Rabm)=(0:_Rabm)\cup(N:_Ram)\cup(N:_Rbm)$.

(2) Apply part (1).
\end{proof}

\begin{theorem}\label{multiplication}
Let $R$ be a $um$-ring, $M$ be a faithful multiplication $R$-module and $N$ be a proper submodule of $M$.
The following conditions are equivalent:

\begin{enumerate}
\item $N$ is a weakly classical prime submodule of $M$;

\item If $0\neq N_1N_2N_3\subseteq N$ for some submodules $N_1,N_2,N_3$ of $
M$, then either $N_1N_3\subseteq N$ or $N_2N_3\subseteq N$;

\item If $0\neq N_1N_2\subseteq N$ for some submodules $N_1,N_2$ of $M$,
then either $N_1\subseteq N$ or $N_2\subseteq N$;

\item $N$ is a weakly prime submodule of $M$;

\item $(N:_RM)$ is a weakly prime ideal of $R$.
\end{enumerate}
\end{theorem}

\begin{proof}
(1)$\Rightarrow$(2) Let $0\neq N_1N_2N_3\subseteq N$ for some submodules $N_1,N_2,N_3$
of $M$. Since $M$ is multiplication, there are ideals $I_1,I_2$ of $R$ such that $N_1=I_1M$ and
$N_2=I_2M$. Therefore $0\neq I_1I_2N_3\subseteq N$, and 
so by Theorem \ref{main2}, $I_1N_3\subseteq N$ or $I_2N_3\subseteq N$.
Thus, either $N_1N_3\subseteq N$ or $N_2N_3\subseteq N$.\newline
(2)$\Rightarrow$(3) Is easy.\newline
(3)$\Rightarrow$(4) Suppose that $0\neq IK\subseteq N$ for some ideal 
$I$ of $R$ and some submodule $K$ of $M$. It is sufficient to set 
$N_1:=IM$ and $N_2=K$ in part (3).\newline
(4)$\Rightarrow$(1) By part (1) of Proposition \ref{abs-class}.\newline
(1)$\Rightarrow$(5) By Proposition \ref{faith}. \newline
(5)$\Rightarrow$(4) Let $0\neq IK\subseteq N$ for some ideal $I$ of $R$ and some
submodule $K$ of $M$. Since $M$ is multiplication, then there is an ideal 
$J$ of $R$ such that $K=JM$. Hence $0\neq IJ\subseteq (N:_{R}M)$ which
implies that either $I\subseteq (N:_{R}M)$ or $J\subseteq {(N:_{R}M)}$.
If $I\subseteq (N:_{R}M),$ then we are done. So, suppose
that $J\subseteq{(N:_{R}M)}$. Thus $K=JM\subseteq N$. 
\end{proof}

\begin{proposition}
\ Let $R$ be a $um$-ring. Let $M$ be a finitely generated faithful
multiplication $R$-module and $N$ a submodule of $M.$ Then the following
conditions are equivalent:
\begin{enumerate}
\item $N$ is a weakly classical prime submodule;

\item $\left( N:_RM\right) $ is a weakly prime ideal of $R$;

\item $N=IM$ for some weakly prime ideal $I$ of $R.$
\end{enumerate}
\end{proposition}
\begin{proof}
$\left( 1\right) \Leftrightarrow \left( 2\right) .$ By Theorem \ref{multiplication}.\newline
$\left( 2\right) \Rightarrow \left( 3\right) $ \ Since $\left( N:_RM\right) $
is a weakly prime ideal and $N=\left( N:_RM\right) M,$ then condition $\left(
3\right) $ holds. \newline
$\left( 3\right) \Rightarrow \left( 2\right) $ Suppose that $N=IM$ \ for
some weakly prime ideal $I$ of $R.$ Since $M$ is a multiplication module, 
we have $N=\left( N:M\right) M.$ Therefore $N=IM=\left( N:M\right) M$ and
so $I=\left( N:M\right) $, because by \cite[ Corollary to Theorem 9]{S} $M$ is cancellation.
\end{proof}

\begin{theorem}\label{prod1}
Let $M_{1}, M_{2}$ be $R$-modules and
$N_{1}$ be a proper submodule of $M_1$. Then the following conditions are equivalent:
\begin{enumerate}
\item $N=N_{1}\times M_{2}$ is a weakly classical prime submodule of $M=M_{1}\times M_{2}$;

\item $N_1$ is a weakly classical prime submodule of $M_1$ and for each $r,s\in R$ and $m_1\in M_1$
we have $$rsm_1=0,\ rm_1\notin N_1, \ sm_1\notin N_1\Rightarrow rs\in Ann_R(M_2).$$
\end{enumerate}
\end{theorem}
\begin{proof}
(1)$\Rightarrow$(2) Suppose that $N=N_{1}\times M_{2}$ is a weakly classical prime submodule of 
$M=M_{1}\times M_{2}$. Let $r,s\in R$ and $m_1\in M_1$ be such that $0\neq rsm_1\in N_1$.
Then $(0,0)\neq rs(m_1,0)\in N$. Thus $r(m_1,0)\in N$ or $s(m_1,0)\in N$, and so $rm_1\in N_1$ or $sm_1\in N_1$.
Consequently $N_1$ is a weakly classical prime submodule of $M_1$. Now, assume that $rsm_1=0$
for some $r,s\in R$ and $m_1\in M_1$ such that $rm_1\notin N_1$ and $sm_1\notin N_1$.
Suppose that $rs\notin{\rm Ann}_R(M_2)$. Therefore there exists $m_2\in M_2$ such that $rsm_2\neq 0$.
Hence $(0,0)\neq rs(m_1,m_2)\in N$, and so $r(m_1,m_2)\in N$ or $s(m_1,m_2)\in N$.
Thus $rm_1\in N_1$ or $sm_1\in N_1$ which is a contradiction. Consequently $rs\in{\rm Ann}_R(M_2)$.\newline
(2)$\Rightarrow$(1) Let $r,s\in R$ and $(m_1,m_2)\in M=M_{1}\times M_{2}$ be such that 
$(0,0)\neq rs(m_1,m_2)\in N=N_{1}\times M_{2}$. First assume that $rsm_1\neq0$. Then by part (2),
$rm_1\in N_1$ or $sm_1\in N_1$. So $r(m_1,m_2)\in N$ or $s(m_1,m_2)\in N$, and thus we are done.
If $rsm_1=0$, then $rsm_2\neq0$. Therefore $rs\notin{\rm Ann}_R(M_2)$, and so part (2) implies that
either $rm_1\in N_1$ or $sm_1\in N_1$. Again we have that $r(m_1,m_2)\in N$ or $s(m_1,m_2)\in N$ which 
shows $N$ is a weakly classical prime submodule of $M$.
\end{proof}

The following two propositions have easy verifications.
\begin{proposition}
Let $M_{1}, M_{2}$ be $R$-modules and $N_{1}$ be a proper submodule of $M_1$. Then 
$N=N_{1}\times M_{2}$ is a classical prime submodule of $M=M_{1}\times M_{2}$ if and only if
$N_1$ is a classical prime submodule of $M_1$.
\end{proposition}


\begin{proposition}\label{prod3}
Let $M_{1}, M_{2}$ be $R$-modules and $N_{1}, N_2$ be  proper submodules of $M_1,M_2$, respectively. If 
$N=N_{1}\times N_{2}$ is a weakly classical prime (resp. classical prime) submodule of $M=M_{1}\times M_{2}$, then
$N_1$ is a weakly classical prime (resp. classical prime) submodule of $M_1$ and $N_2$ is a 
weakly classical prime (resp. classical prime) submodule of $M_2$.
\end{proposition}


\begin{example}
Let $R=\mathbb{Z}$, $M=\mathbb{Z}\times\mathbb{Z}$ and $N=p\mathbb{Z}\times q\mathbb{Z}$ 
where $p,~q$ are two distinct prime integers. Since $p\mathbb{Z},\ q\mathbb{Z}$ are prime ideals 
of $\mathbb{Z}$, then $p\mathbb{Z},\ q\mathbb{Z}$ are weakly classical prime $\mathbb{Z}$-submodules of $\mathbb{Z}$. 
Notice that $(0,0)\neq pq(1,1)=(pq,pq)\in N$, but neither $p(1,1)\in N$ nor
$q(1,1)\in N$. So $N$ is not a weakly classical prime submodule of $M$. 
This example shows that the converse of Proposition \ref{prod3} is not true.
\end{example}

Let $R_i$ be a commutative ring with identity and $M_i$ be an $R_i$-module,
for $i = 1, 2$. Let $R=R_{1}\times R_{2}$. Then $M=M_{1}\times M_{2}$ is an $%
R$-module and each submodule of $M$ is in the form of $N=N_{1}\times N_{2}$ for
some submodules $N_1$ of $M_1$ and $N_2$ of $M_2$.
\begin{theorem}\label{product1}
Let $R=R_{1}\times R_{2}$ be a decomposable ring and $M=M_{1}\times M_{2}$ be an $R$-module where
$M_{1}$ is an $R_{1}$-module and $M_{2}$ is an $R_{2}$-module. Suppose that $N=N_{1}\times M_{2}$ is
a proper submodule of $M$. Then the following conditions are equivalent:
\begin{enumerate}
\item $N_1$ is a classical prime submodule of $M_1$;

\item $N$ is a classical prime submodule of $M$;

\item $N$ is a weakly classical prime submodule of $M$.
\end{enumerate}
\end{theorem}
\begin{proof}
(1)$\Rightarrow$(2) Let $(a_1,a_2)(b_1,b_2)(m_1,m_2)\in N$ for some $(a_1,a_2),(b_1,b_2)\in R$ and $(m_1,m_2)\in M$. 
Then $a_1b_1m_1\in N_1$ so either $a_1m_1\in N_1$ or $b_1m_1\in N_1$ which shows that either $(a_1,a_2)(m_1,m_2)\in N$
or $(b_1,b_2)(m_1,m_2)\in N$. Consequently $N$ is a classical prime submodule of $M$.\newline
(2)$\Rightarrow$(3) It is clear that every classical prime submodule is a weakly classical prime submodule.\newline
(3)$\Rightarrow$(1) Let $abm\in N_1$ for some $a,b\in R_1$ and $m\in M_1$. We may assume that $0\neq m'\in M_2$.
Therefore $0\neq(a,1)(b,1)(m,m')\in N$. So either $(a,1)(m,m')\in N$ or $(b,1)(m,m')\in N$. Therefore
$am\in N_1$ or $bm\in N_1$. Hence $N_1$ is a classical prime submodule of $M_1$.
\end{proof}

\begin{proposition}\label{product2}
Let $R=R_{1}\times R_{2}$ be a decomposable ring and $M=M_{1}\times M_{2}$ be an $R$-module where $M_{1}$
is an $R_{1}$-module and $M_{2}$ is an $R_{2}$-module. Suppose that $N_1,N_2$ are proper submodules
of $M_1,M_2$, respectively. If $N=N_{1}\times
N_{2}$ is a weakly classical prime submodule of $M$, then $N_1$ is a weakly
prime submodule of $M_1$ and $N_2$ is a weakly prime submodule of $M_2$.
\end{proposition}

\begin{proof}
Suppose that $N=N_{1}\times N_{2}$ is a weakly classical prime submodule of $%
M$. By hypothesis, there exist $x\in M_1\backslash N_1$ and $y\in M_2\backslash N_2$.
First we show that $N_1$ is a weakly prime submodule of $M_1$. 
Let $0\neq am_1\in N_{1}$ for some $a\in R_{1}$ and $m_1\in M_{1}$. Then $0\neq
\left( 1,0\right)\left(a,1\right) \left(m_1,y\right) \in N_{1}\times N_{2}=N
$. Notice that if $\left( a,1\right) \left( m_1,y\right) \in N_{1}\times N_{2}=N$, 
then $y\in N_2$ which is a contradiction. So we get $\left(1,0\right)\left(m_1,y\right) \in N_{1}\times N_{2}=N$. 
Thus $m_1\in N_{1}$. Hence $N_{1}$ is a weakly prime submodule
of $M_{1}$. A similar argument shows that $N_{2}$ is a weakly prime
submodule of $M_{2}$.
\end{proof}

The following example shows that the converse of Proposition \ref{product2} is not true in general.
\begin{example}
Let $R=M=\mathbb{Z}\times\mathbb{Z}$ and $N=p\mathbb{Z}\times q\mathbb{Z}$ 
where $p,~q$ are two distinct prime integers. Since $p\mathbb{Z},\ q\mathbb{Z}$ are prime ideals of $\mathbb{Z}$,
then $p\mathbb{Z},\ q\mathbb{Z}$ are weakly prime (weakly classical prime) $\mathbb{Z}$-submodules of $\mathbb{Z}$. 
Notice that $(0,0)\neq(p,1)(1,q)(1,1)=(p,q)\in N$, but neither $(p,1)(1,1)\in N$ nor
$(1,q)(1,1)\in N$. So $N$ is not a weakly classical prime submodule of $M$. 
\end{example}

\begin{theorem}\label{product3}
Let $R=R_{1}\times R_{2}\times R_3$ be a decomposable ring and $M=M_{1}\times M_{2}\times M_3$ be an $R$-module where
$M_{1}$ is an $R_{1}$-module, $M_{2}$ is an $R_{2}$-module and $M_{3}$ is an $R_{3}$-module. 
If $N$ is a weakly classical prime submodule of $M$, then either $N=\{(0,0,0)\}$ or $N$ is a classical prime submodule of $M$.
\end{theorem}
\begin{proof}
Since $\{(0,0,0)\}$ is a weakly classical prime submodule in any module, we may assume
that $N=N_{1}\times N_{2}\times N_3\neq\{(0,0,0)\}$. We assume that $N$ is not a classical prime submodule of $M$
and reach a contradiction. Without loss of generality we may assume that $N_1\neq0$ and so there is $0\neq n\in N_1$.
We claim that $N_2=M_2$ or $N_3=M_3$. Suppose that there are $m_2\in M_2\backslash N_2$ and $m_3\in M_3\backslash N_3$.
Get $r\in(N_2:_{R_2}M_2)$ and $s\in(N_3:_{R_3}M_3)$. Since
$(0,0,0)\neq(1,r,1)(1,1,s)(n,m_2,m_3)=(n,rm_2,sm_3)\in N$, then $(1,r,1)(n,m_2,m_3)=(n,rm_2,m_3)\in N$
or $(1,1,s)(n,m_2,m_3)=(n,m_2,sm_3)\in N$. Therefore either $m_3\in N_3$ or $m_2\in N_2$, a contradiction.
Hence $N=N_{1}\times M_{2}\times N_3$ or $N=N_{1}\times N_{2}\times M_3$. Let $N=N_{1}\times M_{2}\times N_3$.
Then $(0,1,0)\in(N:_RM)$. Clearly $(0,1,0)^2N\neq\{(0,0,0)\}$. So $(N:_{R}M)^{2}N\neq\{(0,0,0)\}$ which
is a contradiction, by Theorem \ref{T2}. In the case when $N=N_{1}\times N_{2}\times M_3$ we have that
$(0,0,1)\in(N:_RM)$ and similar to the previous case we reach a contradiction.
\end{proof}

\begin{theorem}\label{flat}
Let $R$ be a $um$-ring and $M$ be an $R$-module.
\end{theorem}
\begin{enumerate}
\item If $F$ is a flat $R$-module and $N$ is a weakly classical prime 
submodule of $M$ such that $F\otimes N\neq F\otimes M,$ then $%
F\otimes N$ is a weakly classical prime submodule of $F\otimes M.$

\item Suppose that $F$ is a faithfully flat $R$-module. Then $N$ is a
weakly classical prime submodule of $M$ if and only if $F\otimes N$ is a
weakly classical prime submodule of $F\otimes M.$
\end{enumerate}
\begin{proof}
$\left( 1\right) $ Let $a,b\in R$. Then by Theorem \ref{main2}, either $\left(
N:_Mab\right) =\left( 0:_Mab\right) $ or $\left( N:_Mab\right) =\left(
N:_Ma\right) $ or $\left( N:_Mab\right) =\left( N:_Mb\right) $. Assume that $
\left( N:_Mab\right) =\left(0:_Mab\right) $. Then by \cite[Lemma 3.2]{A}, 
$$\left( F\otimes N:_{F\otimes M}ab\right) =F\otimes \left( N:_Mab\right) =F\otimes \left( 0:_Mab\right)
$$$$\hspace{3cm}=\left( F\otimes 0:_{F\otimes M}ab\right)=\left( 0:_{F\otimes M}ab\right).$$
Now, suppose that $\left( N:_Mab\right) =\left(N:_Ma\right)$. Again by \cite[Lemma 3.2]{A}, 
$$\left( F\otimes N:_{F\otimes M}ab\right) =F\otimes \left( N:_Mab\right) =F\otimes \left( N:_Ma\right)
$$$$\hspace{0.7cm}=\left( F\otimes N:_{F\otimes M}a\right).$$
Similarly, we can show that if $\left( N:_Mab\right) =\left( N:_Mb\right) $,
then $\left( F\otimes N:_{F\otimes M}ab\right)=\left( F\otimes N:_{F\otimes M}b\right).$
Consequently by Theorem \ref{main2} we deduce that $F\otimes N$ is a
weakly classical prime submodule of $F\otimes M.$

$\left( 2\right) $ Let $N$ be a weakly classical prime submodule of $M$
and assume that $F\otimes N=F\otimes M$. Then $0\rightarrow F\otimes N%
\overset{\subseteq }{\rightarrow }F\otimes M\rightarrow 0$ is an exact
sequence. Since $F$ is a faithfully flat module, $0\rightarrow N\overset{%
\subseteq }{\rightarrow }M\rightarrow 0$ is an exact sequence. So $N=M$,
which is a contradiction. So $F\otimes N\neq F\otimes M$. Then $F\otimes N$
is a weakly classical prime submodule by $\left( 1\right) $. Now for the
converse, let $F\otimes N$ be a weakly classical prime submodule of $%
F\otimes M$. We have $F\otimes N\neq F\otimes M$ and so $N\neq M$. Let $%
a,b\in R$. Then $\left( F\otimes N:_{F\otimes M}ab\right) =\left( 0:_{F\otimes M}ab\right) $  
or $\left( F\otimes N:_{F\otimes M}ab\right) =\left( F\otimes
N:_{F\otimes M}a\right) $ or $\left( F\otimes N:_{F\otimes M}ab\right) =\left( F\otimes N:_{F\otimes M}b\right) $
by Theorem \ref{main2}. Suppose that $\left( F\otimes N:_{F\otimes M}ab\right) =\left( 
0:_{F\otimes M}ab\right)$. Hence 
$$F\otimes \left( N:_Mab\right)=\left( F\otimes N:_{F\otimes M}ab\right)=\left(
0:_{F\otimes M}ab\right)
$$$$\hspace{2.6cm}=\left( F\otimes 0:_{F\otimes M}ab\right)=F\otimes \left( 0:_Mab\right).$$
Thus  $0\rightarrow F\otimes \left( 0:_Mab\right) \overset{\subseteq }{%
\rightarrow }F\otimes \left( N:_Mab\right) \rightarrow 0$ is an exact
sequence. Since $F$ is a faithfully flat module, $0\rightarrow \left(
0:_Mab\right) \overset{\subseteq }{\rightarrow }\left( N:_Mab\right)
\rightarrow 0$ is an exact sequence which implies that $\left( N:_Mab\right)
=\left( 0:_Mab\right) $. With a similar argument we can deduce that if 
$\left( F\otimes N:_{F\otimes M}ab\right) =\left( F\otimes N:_{F\otimes M}a\right) $ or 
$\left( F\otimes N:_{F\otimes M}ab\right) =\left( F\otimes N:_{F\otimes M}b\right)$, then 
$\left( N:_Mab\right)=\left( N:_Ma\right)$ or $\left( N:_Mab\right)=\left( N:_Mb\right) $.
Consequently $N$ is a weakly classical prime
submodule of $M$ by Theorem \ref{main2}.
\end{proof}

\begin{corollary}
Let $R$ be a $um$-ring, $M$ be an $R$-module and $X$ be an indeterminate. If $N$ is a weakly classical prime submodule of $M$,
then $N[X]$ is a weakly classical prime submodule of $M[X]$.
\end{corollary}
\begin{proof}
Assume that $N$ is a weakly classical prime submodule of $M$. Notice that $R[X]$ is a flat $R$-module.
Then by Theorem \ref{flat}, $R[X]\otimes N\simeq N[X]$ is a weakly classical prime submodule of $R[X]\otimes M\simeq M[X]$.
\end{proof}

\vspace{3mm} \noindent \footnotesize 
\begin{minipage}[b]{10cm}
Hojjat Mostafanasab \\
Department of Mathematics and Applications, \\
University of Mohaghegh Ardabili, \\
P. O. Box 179, Ardabil, Iran. \\
Email: h.mostafanasab@uma.ac.ir, \hspace{1mm} h.mostafanasab@gmail.com
\end{minipage}

\vspace{3mm} \noindent \footnotesize
\begin{minipage}[b]{10cm}
\"{U}nsal Tekir\\
Department of Mathematics, \\ 
Marmara University, \\ 
Ziverbey, Goztepe, Istanbul 34722, Turkey. \\
Email: utekir@marmara.edu.tr
\end{minipage}

\vspace{3mm} \noindent \footnotesize
\begin{minipage}[b]{10cm}
K\"{u}r\c{s}at Hakan Oral\\
Department of Mathematics,\\ 
Yildiz Technical University,\\ 
Davutpasa Campus, Esenler, Istanbul, Turkey.\\
Email: khoral@yildiz.edu.tr
\end{minipage}

\end{document}